\documentclass[a4paper,draft,reqno,12pt]{amsart}
\usepackage[english]{babel}
\usepackage{amsmath}
\usepackage{amssymb}
\usepackage{amscd}
\usepackage{amsthm}
\usepackage{euscript}
\usepackage{tikz}
\newtheorem{proposition}{Proposition}
\newtheorem{lemma}{Lemma}
\newtheorem{theorem}{Theorem}
\newtheorem{question}{Question}
\newtheorem{corollary}{Corollary}
\theoremstyle{definition}
\newtheorem{definition}{Definition}

\newtheorem{conjecture}{Conjecture}
\theoremstyle{remark}
\newtheorem {remark}{Remark}

\DeclareMathOperator{\Aut}{Aut}

\def\Ker{{\rm Ker}\,}

\def\Gr{{\rm Gr}\,}

\def\ZZ{{\mathbb Z}}
\def\BG{{\mathbb G}}
\def\BK{{\mathbb K}}

\def\BG{{\mathbb G}}

\def\BK{{\mathbb K}}

\def\BZ{{\mathbb Z}}
\def\BR{{\mathbb R}}

\def\BN{{\mathbb N}}

\def\BA{{\mathbb A}}

\def\ML{\mathrm{ML}}
\def\HD{\mathrm{HD}}

\def\SAut{\mathrm{SAut}}

\def\LND{\mathrm{LND}}

\title{Modified Derksen invariant}

\thanks{The paper was supported by RSF grant 22-41-02019.}

\author{Ilya Boldyrev}
\email{boldyrev.i.al@gmail.com}
\address{
Lomonosov Moscow State University, Faculty of Mechanics and Mathematics, Department of Higher Algebra, Leninskie Gory 1, Moscow, 119991 Russia;
\linebreak
and
\linebreak
HSE University, Faculty of Computer Science, Pokrovsky Boulevard 11, Moscow, 109028, Russia}

\author{Sergey Gaifullin}
\email{sgayf@yandex.ru}
\address{Moscow Center for Fundamental and Applied Mathematics, Moscow, Russia
\linebreak
and
\linebreak
Lomonosov Moscow State University, Faculty of Mechanics and Mathematics, Department of Higher Algebra, Leninskie Gory 1, Moscow, 119991 Russia;
\linebreak
and
\linebreak
HSE University, Faculty of Computer Science, Pokrovsky Boulevard 11, Moscow, 109028, Russia}
\author{Anton Shafarevich}
\email{shafarevich.a@gmail.com}
\address{Moscow Center for Fundamental and Applied Mathematics, Moscow, Russia
\linebreak
and
\linebreak
Lomonosov Moscow State University, Faculty of Mechanics and Mathematics, Department of Higher Algebra, Leninskie Gory 1, Moscow, 119991 Russia;
\linebreak
and
\linebreak
HSE University, Faculty of Computer Science, Pokrovsky Boulevard 11, Moscow, 109028, Russia}

\begin{document}
\maketitle

\begin{abstract}
The modified Derksen invariant  $\HD^*(X)$ of an affine algebraic variety $X$ is the subalgebra in $\BK[X]$ generated by kernels of all locally nilpotent derivations of $\BK[X]$ with slices. If there is a locally nilpotent derivation of $\BK[X]$ with a slice then $X \simeq Y \times \BA^1$ where $Y$ is an affine variety. We prove that there are three possibilities: A) $\HD^*(X) = 
\BK[X];$ B) $\HD^*(X) $ is a proper infinitely generated subalgebra; C) $\HD^*(X) = \BK[Y].$ We give examples for each case, and also provide sufficient conditions for the variety $Y$ so that the variety $X = Y \times \BA^1$ belongs to one of the types.
\end{abstract}

\section{Introduction}

Let $\BK$ be an algebraically closed field of characteristic zero and $A$ be a $\BK$-domain. A linear map $\partial\colon A\rightarrow A$ is called a \emph{derivation} if it satisfies the Leibniz rule: $\partial(ab)=a\partial(b)+b\partial(a)$. A derivation is called {\it locally nilpotent} (LND) if for any $a \in A$ there exists $n\in \mathbb{N}$ such that $\partial^n(a)=0$. We denote the set of all LNDs on $A$ by $\LND(A).$ If $A = \BK[X]$ for an irreducible affine variety $X$ then by $\LND(X)$ we mean $\LND(\BK[X]).$

Each locally nilpotent derivation $\partial$ corresponds to an action of the group $\BG_a = (\BK, +)$ on $A$:
$$s\cdot a = \mathrm{exp}(s\partial)(a) = \sum_{i=0}^{+\infty} \frac{s^i\partial^i(a)}{i!},$$
where $s \in \BG_a,\ a\in A.$ Here we mean $\partial^0=\mathrm{id}$. The sum is well-defined since for each $a$ only a finite number of terms are nonzero. If $X$ is an affine algebraic variety and  $A = \BK[X]$ is the algebra of regular functions on $X$ then this $\BG_a$-action on $\BK[X]$ is regular and defines regular $\BG_a$-action on $X$. This correspondence between $\LND(X)$ and regular $\BG_a$-actions on $X$ is bijective, see \cite{Fr}. For $\partial \in \LND(X)$ by $\exp(\partial)$ we mean the respective $\BG_a$-subgroup in $\Aut(X)$ where $\Aut(X)$~--- the group of regular automorphisms of $X$.

The Makar-Limanov invariant is the  subalgebra in $\BK[X]$ equal to the intersection of  kernels of all LNDs on $X$:
$$
\ML(X) = \bigcap_{\partial \in \mathrm{LND}(X)} \mathrm{Ker}\,\partial.
$$
The subalgebra $\ML(X)$ is $\Aut(X)$-invariant, so it is called the \emph{Makar-Limanov invariant} of~$X$.
The Makar-Limanov invariant was introduced in \cite{ML96} and became a useful tool to study affine varieties, see~\cite{ML96, KML, ML01, MJ, P, G}.

The Derksen invariant was introduced in \cite{HD}. It is the subalgebra in $\BK[X]$ generated by kernels of all nonzero LNDs:
$$
\HD(X)=\BK\left[\mathrm{Ker}\,\partial\mid\partial\in \mathrm{LND}(X)\!\setminus\!\{0\}\right].
$$

In \cite{CM} examples of varieties were constructed for which the Makar-Limanov invariant is trivial, but the Derksen invariant is not, and vice versa. Sometimes for varieties, which we would like to investigate, both invariants $\ML(X)$ and $\HD(X)$ are trivial. For example, let us consider a variety $X\cong Y\times \BA^1$ for some irreducible affine variety $Y$. Each LND $\partial$ on $Y$ can be extended to an LND on $X$ by $\partial(u)=0$, where~$u$ is the additional coordinate. Using this it can be shown that if $Y$ admits a nonzero LND, then $\HD(X)=\BK[X]$, and if $\ML(Y)=\BK$, then $\ML(X)=\BK$. So, to deal with such varieties we need other invariants. 

For some LNDs there exist elements, which are called slices.
\begin{definition}
    Let $\partial$ be a locally nilpotent derivation. We say that a function $s\in \BK[X]$ is a \emph{slice} with respect to $\partial$ if $\partial(s) = 1.$ We denote by $\mathrm{LND}^*(X)$ the set of all locally nilpotent derivations of $\BK[X]$ that have a slice.
     
\end{definition}
The well-known Slice theorem, see~\cite[Corollary~1.26]{Fr}, says that if $s$ is a slice of $\partial$ and $A=\mathrm{Ker}\,\partial$, then $s$ is transcendent over $A$ and $\BK[X]=A[s]$. One can show that in this case $A$ is finitely generated algebra. So $X \simeq Y \times \BA^1$, where $Y$ is an affine variety. 

In Section~11.9 of the first edition of the book~\cite{Fr}, Freudenburg suggested to consider the following modifications of Makar-Limanov and Derksen invariants. 
\begin{definition}
    Let $X$ be an affine variety. Suppose that $\mathrm{LND}^*(X) \neq \emptyset$. Then the \emph{modified Makar-Limanov} invariant is the algebra:
    $$\ML^*(X) = \bigcap_{\partial \in \mathrm{LND}^*(X)} \mathrm{Ker}\ \partial \subseteq \BK[X].$$

    The \emph{modified Derksen invariant} of an affine variety $X$ is the subalgebra in $\mathbb{K}[X]$ generated by all kernels of LNDs having slices:
$$\HD^*(X)=\BK[\mathrm{Ker}\,D\mid D\in \mathrm{LND}^*(X)].$$
\end{definition}

It is proved in~\cite{GSh2} that for any variety $Y$ we have $\ML(Y\times\BA^1)=\ML^*(Y\times\BA^1)$. So, $\ML^*(X)$ does not give us a new invariant. In particular, this gives negative answer to the question stated in~\cite[Question~5.9]{DG}, if the product of Koras-Russell cubic and a line and $\BA^4$ can be distinguished by $\ML^*(X)$. In the same paper an example of a variety with trivial $\ML(X)$, $\ML^*(X)$, $\HD(X)$, and a nontrivial $\HD^*(X)$ is given. So, $\HD^*(X)$ is a new invariant. The question whether $\HD^*$ distinguish the product of Koras-Russell cubic and a line and $\BA^4$ remains open, see~\cite[Question~1]{GSh2}.

In this paper we investigate the invariant $\HD^*(X)$. If $\LND^*(X) \neq \emptyset$ then $X \simeq Y \times \BA^1$ so it is enough to consider varieties of the form $Y\times \BA^1.$ We prove that there are only three cases: \begin{itemize}
    \item[A) ] $\HD^*(Y\times\BA^1)=\BK[Y\times\BA^1]$;
    \item[B) ] $\HD^*(Y\times\BA^1)$ is not finitely generated;
    \item[C) ] $\LND(Y) = \{0\}$ and $\HD^*(Y\times\BA^1)=\BK[Y]$.
\end{itemize} 
We say that $Y$ is of type A, B, or C according to these cases. 

Varieties $Y$ with $\LND(Y) = \{0\}$ (type C) are called \emph{rigid}. Many works are devoted to rigid varieties, see, for example, \cite{AG}, \cite{KZ}, \cite{CM2} and \cite{FMJ}. There are many examples of varieties of type $A$ and $B$. We give some in Section \ref{examples}. We also obtain some sufficient conditions when a variety $Y$ belongs or does not belong to one of the types A, B, C (Propositions \ref{pp}-\ref{nni}).

The second author is a Young Russian Mathematics award winner and would like to thank its sponsors and jury.

\section{Preliminaries}

Let $F$ be an abelian group.

\begin{definition}
An algebra $A$ is called {\it $F$-graded} if $$A=\bigoplus_{f\in F}A_f \text{ and } A_fA_g\subset A_{f+g}.$$
\end{definition}

\begin{definition}
A derivation $\partial\colon A\rightarrow A$ is called {\it $F$-homogeneous of degree $f_0\in F$} if for all $a\in A_f$ we have $\partial(a)\in A_{f+f_0}$.
\end{definition}
\noindent

Let $A$ be a finitely generated $\ZZ$-graded algebra.
\begin{lemma}
Let $\partial$ be a derivation, then $\partial=\sum_{i=l}^k\partial_i$, where $\partial_i$ is a homogeneous derivation of degree $i$. 
\end{lemma}
\begin{proof}
Let $a_1,\ldots,a_m$ be the generators of $A$. Then $\partial (a_j)=\sum_{i=l_j}^{k_j}b_{ij}$, where $b_{ij}\in A_i$. Take $l:=\mathrm{min}\{l_1,\ldots,l_m\}$, $k:=\mathrm{max}\{k_1,\ldots,k_m\}$. Using the Leibniz rule we get
$$\forall s \in \BZ,\ \forall a\in A_s \colon \; \partial (a)\in\bigoplus_{i=l}^k A_{s+i}.$$ Thus, $\partial=\sum_{i=l}^k\partial_i$, where $\partial_i\colon A_s\rightarrow A_{s+i}$ is a linear map. The Leibniz rule for $\partial_s$ follows from the Leibniz rule for $\partial$.
\end{proof}

\begin{remark}
Further, when we write $\partial=\sum_{i=l}^k\partial_i$, we will assume that $\partial_l\neq 0$ and~$\partial_k\neq 0$.
\end{remark}

We need the following lemma. For the proof, we refer to~\cite{Re}.
\begin{lemma}\label{fl}
Let $A$ be a finitely generated $\ZZ$-graded algebra, $\partial\colon A\rightarrow A$ be an~LND. Assume $\partial=\sum_{i=l}^k\partial_i$, where $\partial_i$ is the homogeneous derivation of degree $i$. Then $\partial_l$ and $\partial_k$ are LNDs.
\end{lemma}
\begin{corollary}\label{flz}
If $A$ admits a LND, then $A$ admits a $\ZZ$-homogeneous LND. 
\end{corollary}

\begin{lemma}\label{invlem}
    Let $X$ be an affine variety and $U$ is a subspace in $\BK[X]$ which is invariant with respect to all automorphisms of $X$. Then $\partial(U) \subseteq U$ for all $\partial \in \LND(X)$. If $\BK[X] = \bigoplus_i\BK[X]_i$ is a $\BZ$-graded algebra then 
    $U = \bigoplus_i U\cap \BK[X]_i.$ 
\end{lemma}
\begin{proof}
    Let $\partial \in \LND(X)$ and $f \in U$. Then 
    $$\exp(s\partial)(f) = f + s\partial(f) + \frac{s^2\partial^2(f)}{2} + \ldots + \frac{s^m\partial^m(f)}{m!} \in U$$
    for all $s \in \BK$. Since $\BK$ is an infinite field we obtaind that $\partial^k(f) \in U$ for all $k.$

    Now we suppose that $\BK[X]$ is a $\BZ$-graded algebra: 
    $$\BK[X] = \bigoplus_{i\in \BZ}\BK[X]_i,\ i\in \BZ.$$
    Then we have a $\BK^*$-action on $X$:
    $$t\circ f = t^if,\ t\in \BK^*,\ f\in \BK[X]_i.$$
    The subspace $U$ is invariant with respect to this action. An element $u\in U$ can be represented as
    $$u = \sum_{i\in \BZ} u_i,\ u_i \in \BK[X]_i.$$
    Therefore
    $$t\circ u = \sum_{i\in \BZ} t^iu_i \in U,\ \forall t\in \BK^*.$$
    It implies that $u_i \in U$ for all $i\in \BZ.$
 \end{proof}

\section{Main results}

The main result of this paper is the following theorem.

\begin{theorem}\label{infgen} Let $X$ be an affine irreducible variety. Suppose that $X \simeq Y \times \BA^1,$ where $Y$ is an affine variety with $\LND(Y) \neq \{0\}.$ Suppose that $\HD^*(X) \neq \BK[X].$ Then $\HD^*(X)$ is not finitely generated.

\end{theorem}

\begin{proof}
    
We have $\BK[X] = \BK[Y][u]$ for some element $u\in \BK[X].$ There is a natural $\BZ$-grading on $\BK[X]$:
$$\BK[X] = \bigoplus_i \BK[X]_i,\ \BK[X]_i = \langle u^if \mid f \in \BK[Y]\rangle.$$
The derivation $\frac{\partial}{\partial u}$ have a slice $u$ so $\frac{\partial}{\partial u} \in \LND^*(X).$ The kernel of $\frac{\partial}{\partial u}$ is $\BK[Y]$, hence $\BK[Y] \subseteq \HD^*(X).$ Since $\HD^*(X) \neq \BK[X]$ we have $u \notin \HD^*(X).$

We consider the subspaces
$$J_i = \{f \in \BK[Y] \mid u^if \in \HD^*(X)\}.$$ 
Since $\BK[Y] \subseteq \HD^*(X)$ the subspaces $J_i$ are ideals in $\BK[Y].$ 

By Lemma \ref{invlem} we have $\frac{\partial}{\partial u}(\HD^*(X)) \subseteq \HD^*(X).$ So if $f \in J_{i+1}$ then $$\frac{\partial}{\partial u}(fu^{i+1}) = (i+1)fu^{i} \in \HD^*(X)$$ 
and $f \in J_i.$ So we have $J_{i+1} \subseteq J_{i}.$ Note that the subspace $J_1 \neq \BK[Y].$ Otherwise $u\in \HD^*(X)$ and $\HD^*(X) = \BK[X].$

Now suppose $\HD^*(X)$ is finitely generated. Then one can choose a finite set of homogeneous generators $A$ of $\HD^*(X)$. We choose the maximal $r$ such that $A\cap (u^rJ_r) \neq \emptyset.$
Then $J_{ir} \subseteq (J_1)^i$ for all $i\in \BN.$ So 
$$\bigcap_i J_i \subseteq \bigcap_i (J_1)^i.$$

The ideal $J_1$ is proper ideal in $\BK[Y]$. So there is a point $z\in Y$ such that $f(z) = 0$ for all $f \in J_1.$ Let $m_z$ be the ideal of functions in $\BK[Y]$ that vanish at $z$. Then $J_1 \subseteq m_z.$ Denote by $\overline{m}_z$ the maximal ideal in the local ring $\mathcal{O}_z(Y).$  Then $m_z \subseteq \overline{m}_z.$ So we have
$$\bigcap_i J_i \subseteq \bigcap_i (J_1)^i \subseteq \bigcap_i m_z^i \subseteq \bigcap_i \overline{m}_z^i = \{0\}.$$
The last equation is true due to Proposition 6.4 of appendix in \cite{Shaf}.

Consider an ideal $I = (\mathrm{Im}\ \partial \mid \partial \in \LND(Y))$ in $\BK[Y].$ If $\partial \in \LND(Y)$ we can extend $\partial$ to an LND of $\BK[X]$ by $\partial(u)=0$. Then $u^i\partial \in \LND(X).$ So if $f = \partial(g)$, where $g \in \BK[Y] \subseteq \HD^*(X)$ then $u^i\partial(g) = fu^i \in \HD^*(X).$ It implies that $I \subseteq J_i$ for all $i.$ So we obtain a contradiction:
$$\{0\} \neq I \subseteq \bigcap J_i = \{0\}.$$

\end{proof}

\begin{remark}
    Note that in the proof of Theorem~\ref{infgen} we use only the following properties of $\HD^*(X)$: 
    \begin{enumerate}
        \item $\HD^*(X)$ is $\partial$-invariant for every $\partial\in\LND(X)$;
        \item $\BK[Y]\subseteq \HD^*(X)$;
        \item $\HD^*(X)$ is a homogeneous subalgebra in $\BK[X]$
    \end{enumerate}
    So, in conditions of Theorem~\ref{infgen}, i.e. when $X=Y\times\BA^1$ and $Y$ is not rigid, each subalgebra that sutisfies these conditions is either $\BK[X]$ or infinitely generated. 
\end{remark}

\begin{corollary}\label{alt}
    Let $X$ be an affine irreducible variety and $X = Y \times \BA^1$, where $Y$ is an affine irreducible variety. Then one of the following holds:
    \begin{enumerate}
        \item[(A)] $\HD^*(X) = \BK[X];$
        \item[(B)] $\HD^*(X)$ is a infinitely generated subalgebra;
        \item[(C)] $Y$ is rigid and $\HD^*(X) = \BK[Y]$.
    \end{enumerate}
\end{corollary}
\begin{proof}
    Suppose that $Y$ is rigid.  By theorem due to Makar-Limanov, see~\cite[Theorem~2.24]{Fr} the variety $X=Y \times \BA^1$ is semi-rigid. That is, kernels of all nonzero LNDs coincide. Since $\Ker\,\frac{\partial}{\partial u}=\BK[Y]$, we have $\HD^*(X)=\BK[Y]$. 

    If $Y$ is not rigid, then by Theorem~\ref{infgen} either $\HD^*(X)=\BK[X]$ or it is infinitely generated.
\end{proof}

Let us say that $Y$ is of type $A$, $B$ of $C$ if for this variety the item $A$, $B$ or $C$ respectively of Corollary~\ref{alt} takes place. It is useful to determine the type of a variety. It is useful at least because we can state the following corollary.

\begin{corollary}
If $Y\times \BA^1\cong Z\times \BA^1$, then the types of $Y$ and $Z$ coincide.
\end{corollary}

Our goal is to obtain some sufficient conditions for a variety to be of type A or B.

\begin{proposition}\label{pp}
    If the ideal $I$ of $\BK[Y]$ generated by all images of all LNDs on $Y$ coincide with $\BK[Y]$, then $Y$ is of type A.  
\end{proposition}
\begin{proof}
    In the proof of Theorem~\ref{infgen} we have obtained $I\subseteq J_i$ for all $i$. If $I=\BK[Y]$, then $J_i=\BK[Y]$ for all $i$. That is, $\HD^*(Y\times \BA^1)=\BK[Y\times \BA^1]$.
\end{proof}
\begin{corollary}\label{ba1}
If $Y=V\times \BA^1$ for a variety $V$, then $Y$ is of type A.
\end{corollary}
\begin{proof}
In this condition $\BK[Y]$ admits an LND with a slice, i.e. $1\in I$.
\end{proof}
Each $\BG_a$-action on $Y$ defines a subgroup in $\Aut(Y)$, which is called  $\BG_a$-subgroup. Following \cite{AFKKZ}, we call the subgroup in $\Aut(Y)$ generated by all $\BG_a$-subgroups by the subgroup of {\it special automorphisms} and we denote it by $\SAut(Y)$. Note that the condition of Proposition~\ref{pp} can be reformulated in terms of $\SAut(Y)$-action on $Y$.
\begin{proposition}\label{ppp}
    Suppose there are no $\SAut(Y)$-fixed points on $Y$. Then $Y$ is of type A. 
\end{proposition}
\begin{proof}
    If $Y$ is not of type A, then the ideal $I$ is proper. Therefore, $I$ is contained in a maximal ideal $m$. Let $y\in Y$ be the zero point of $m$. Let $\varphi=\exp \partial$ for some LND $\partial$ of $\BK[Y]$. Let us take $f\in \BK[Y]$. Then 
    $$
    \varphi(f)(y)=\left(f+\partial(f)+\frac{\partial^2(f)}{2!}+\ldots\right)(y)=f(y).
    $$
    This implies $\varphi(y)=y$. So, $y$ is a $\SAut(Y)$-fixed point.
\end{proof}

Now let us give a sufficient condition for a variety to be of type B or C. First of all let us give this condition in terms of $X$.
\begin{proposition}\label{suf}
    Suppose there is a point $y\in Y$ such that the line $Z = \{y\}\times \BA^1 \subseteq X = Y \times \BA^1$ is $\SAut(X)$-invariant. Then $Y$ is not of type A. 
\end{proposition}
\begin{proof}
    Suppose $\partial \in \LND^*(X).$ We denote by $I(Z)\subseteq \BK[X]$ the ideal of all regular functions vanishing on $Z.$ Then $I(Z)$ is invariant with respect to $\partial$. So $\partial$ defines a LND $\delta$ of $\BK[Z] = \BK[X]/I(Z) \simeq \BK[t].$ Note that $\delta$ also have a slice so it is a nonzero LND.  Kernels of all nonzero LNDs of $\BK[t]$ coincide with $\BK$. So $\Ker \delta = \BK$. 

    Denote by $\pi$ the canonical homomorphism 
    $$\pi: \BK[X] \to \BK[Z],\ f \to f + I(Z).$$
    Then $\pi(\Ker \partial) \subseteq \Ker \delta = \BK.$ So $\Ker \partial \subseteq \pi^{-1}(\BK).$ But $\pi^{-1}(\BK)$ is a proper subalgebra in $\BK[X].$ So $\HD^*(X) \neq \BK[X].$

\end{proof}

\begin{proposition}\label{nni}
    Suppose there is a nonempty rigid closed subset $V\subset Y$ such that $V\times \BA^1\subset X=Y\times\BA^1$ is a $\SAut(X)$-invariant set. Then $Y$ is  not of type A. 
\end{proposition}
\begin{proof}
Consider $\partial\in\LND(X)$. Since $V\times \BA^1\subset X$ is a $\SAut(X)$-invariant set, $\partial$ induces an LND $\delta$ of $\BK[V\times \BA^1]=\BK[V][u]$. Let us consider the gradings $\BK[V][u]=\bigoplus\limits_{i\geq 0}\BK[V]u^i$ and $\BK[X]=\BK[Y][u]=\bigoplus\limits_{i\geq 0}\BK[Y]u^i$. 

The algebra $\BK[Y][u]$ generated by $\BK[Y]$ and $u$. So there is no nonzero LNDs with degree $\leq -2$. Then $\partial=\partial_{-1}+\partial_0+\ldots+\partial_k$. Since the ideal $I(V\times \BA^1)$ is $\partial$-invariant and homogeneous, it is $\partial_i$-invariant. So each $\partial_i$ induces a derivation $\delta_i$ on $\BK[V][u]$ and we have
$\delta=\delta_{-1}+\delta_0+\ldots+\delta_k$. Let $l$ be the maximal number such that $\delta_l \neq 0.$ Since $\delta$ is LND, the derivation $\delta_l$ is also a nonzero LND. 

If $l\geq 0$, then $\delta_l(u)$ is divisible by $u$. By \cite[Corollary 1.23]{Fr} we have $\delta_l(u)=0$. 
Hence, $\delta_l=u^l\widetilde{\delta}$, where $\widetilde{\delta}$ is an LND of $\BK[V]$ induced by $\frac{\delta_l}{u^l}$. The conditions of the Proposition implies $\delta_l=0$. This is a contradiction. 

Therefore, $\delta=\delta_{-1}=f\frac{\partial}{\partial u}$. Then for each point $v\in V$ we have $\{v\}\times\BA^1$ is $\exp(\delta)$-invariant in $V\times \BA^1$. Hence, $\{v\}\times\BA^1$ is $\exp(\partial)$-invariant in $X$. So, $\{v\}\times\BA^1$ is $\SAut(X)$-invariant. By Proposition~\ref{suf} the variety $Y$ is not of type A.
\end{proof}

The conditions of Proposition~\ref{nni} can be achieved by geometrical reasons. Let $Y$ be a variety of dimension $n$. For a positive integer $k$ we denote by $Y_k$ the set of $y\in Y$ such that the dimension of the tangent space in $y$ is not less than $k$. The subset $Y_k\times \BA^1 \subseteq X = Y \times \BA^1$ is $\Aut(X)$-invariant. So we obtain the following statement.
\begin{corollary}\label{sinsin}
    Suppose there exists positive integer $k$ such that $Y_k$ is rigid. Then $Y$ is of type B or C. 
\end{corollary}

\section{Examples}\label{examples}

Let us give some examples of varieties of type A. All of them are examples of varieties satisfying conditions of Proposition~\ref{pp} and~\ref{ppp}.

\begin{enumerate}
\item $V\times \BA^1$ for any affine irreducible variety $V$, see Corollary~\ref{ba1}.
\item Danielewski surfaces $W_n=\{x^ny=f(z)\}\subseteq \BA^3$, $n\geq 1$, where the polynomial $f$ has no multiple roots. There is an LND $\partial$ on $W_n$ given by $\partial(x)=0, \partial(y)=f'(z), \partial(z)=x^n$. Then $f'(z), y^n\in I$, where $I$ is an ideal generated by images of LNDs. Hence, $f(z)=y^nx\in I$ and $1=\gcd(f,f')(z)\in I$. By Proposition \ref{pp}, the variety $W_n$ is of type A.
\item Danielewski varieties of the form $\{x_1^{k_1}\ldots x_m^{k_m}y=f(z)\} \subseteq\BA^{m+2}$, where the polynomial $f$ has no multiple roots. These varieties admits LNDs of the form 
$$\partial(x_i)=0,\qquad \partial(y)=f'(z), \qquad \partial(z)=x_1^{k_1}\ldots x_m^{k_m}.$$
One can check that the corresponding $\BG_a$-actions have no fixed points. 
\item Smooth flexible varieties, i.e. varieties $Y$ with unique $\mathrm{SAut}(Y)$-orbits. Examples of such varieties are smooth affine varieties with actions of reductive groups with an open orbits, see the proof in~\cite{GSh}.
\item Varieties $Y$ with actions of semi-simple algebraic group $G$ without fixed points. Indeed, semi-simple algebraic group is generated by unipotent subgroups. Hense $G\subseteq \SAut(Y)$. This implies that there are no $\SAut(Y)$-fixed points.
\item Let $Z$ be a variety of type A and $V$ be an arbitrary irreducible variety. Then $Y=V\times Z$ is of type A. Indeed, we can extend each LND $\partial$ on $\BK[Z\times \BA^1]$ to $\BK[Z\times \BA^1]\otimes \BK[V]$ by $\partial(f)=0$ for $f\in \BK[V]$. Kernels of such LNDs generate $\BK[Z\times V\times \BA^1]$. 
\item 
Let $Z$ be an irreducible affine variety and $f\in\BK[Z]$ is not a constant. By {\it suspension} over $Z$ we mean the subvariety in $Z\times \BA^2$ given by $uv=f$, where $u$ and $v$ are coordinates on $\BA^2$. LNDs on suspensions were investigated in~\cite{AKZ}, see also~\cite{AZ}. In~\cite{G} the following generalization of this concept was introduced. By {$m$-suspension} with positive integer weights $k_1,\ldots, k_m$ we mean the subvariety $\mathrm{Susp}(Z,f,k_1,\ldots, k_m)$ in $Z\times\BA^m$ given by 
$$
y_1^{k_1}\ldots y_m^{k_m}=f,
$$
where $y_1,\ldots, y_m$ are coordinates on $\BA^m$. Suppose $k_1=1$, the ideal $J$ in $\BK[Z]$ generated by all images of LNDs is not proper, and the ideal $I_f$ in $\BK[Z]$ generated by $f$ and all images of $f$ under all LNDs is not proper. Let us prove that $Y=\mathrm{Susp}(Z,f,k_1,\ldots, k_m)$ is of type A. Indeed, if $\partial$ is an LND of $\BK[Z]$, then we can define an LND $\delta$ of $\BK[Y]$ by $\delta(h)=\partial(h)y_2^{k_2}\ldots y_m^{k_m}$ for $h\in \BK[Z]$ , $\delta(y_1)=\partial(f)$, $\delta(y_2)=\ldots=\delta(y_m)=0$. Now, we see that the ideal $I$ in $\BK[Y]$ generated by all images of LNDs contains $\partial(f)$ and $y_2^{k_2}\ldots y_m^{k_m}$. Since  $y_1^{k_1}\ldots y_m^{k_m}=f$, we have $f\in I$. So, $1\in I_f\subseteq I$. Then by Proposition~\ref{pp}, the variety $Y$ is of type~A.

Note that Danielewski surfaces and Danielewski varieties are particular cases of varieties considered in this item.

\end{enumerate}

Now let us give some examples of varieties of type B.

\begin{enumerate}
\item Let $Y$ be a non-rigid variety. Suppose we have $\mathbb{Z}$-grading of $\BK[Y]$ such that $\BK[Y]_i=\{0\}$ for $i<0$. Assume that $\BK[Y]$ does not allow homogeneous LNDs with negative degrees and for each homogeneous LND $\partial$ of degree $0$, the restriction $\partial$ to $\BK[Y]_0$ is zero. Let us show that $Y$ has type B or C. Denote by $J$ the sum of all positive homogeneous components: $J=\bigoplus_{i>0}\BK[Y]_i$. Then $J$ is an ideal. Denote the zero set of $J$ by $V$. Then $\BK[V]\cong \BK[Y]_0$. Our conditions imply that $V$ is a rigid variety. Let us prove that $V\times\BA^1$ is $\SAut(X)$-invariant, where $X = Y \times \BA^1$. Note that $\BK[X]$ admits $\mathbb{Z}^2$-grading: $\BK[X]_{ij}=\BK[Y]_iu^j$. If $\delta$ is a $\mathbb{Z}^2$-homogeneous LND of degree $(a,b)$ and $b<0$, then $\delta=f\frac{\partial}{\partial u}$ and $a\geq 0$. If $b\geq 0$, then $\delta(u)=0$ and $\delta=u^b\xi$, where $\xi$ is an LND of $\BK[Y]$. Therefore, $a\geq 0$. If $\partial$ is an LND of $\BK[X]$, we can decompose $\partial$ onto homogeneous summands with respect to the first $\mathbb{Z}$-grading: $\partial=\partial_l+\ldots+\partial_k$. Then $\partial_l$ is an LND and hence $l\geq 0$. It implies that the ideal $I(V\times\BA^1)$ is $\SAut(X)$-invariant. 
By Proposition~\ref{nni} the variety $Y$ is of type B.
\item A particular case of the previous item is non-rigid non-degenerate affine toric varieties $Y$ without line factors, i.e. $Y\ncong Z\times \BA^1$. 

Let $Y$ be an affine toric variety and $T$ is a torus which acts effectively on $Y$ with an open orbit. Let $M$ be the lattice of characters of $T$ and $N$ be the dual lattice of one-parameter subgroups in $T$. By $\langle \cdot, \cdot \rangle$ we denote the natural pairing $M\times N \to \mathbb{Z}.$ Denote by $\sigma$ the cone in $N_{\BR} = N \otimes_{\BZ} \BR$ which corresponds to $Y$ and $v_1,\ldots, v_m \in N$~--- primitive integer vectors on rays of $\sigma.$ 

Denote by $\sigma^{\vee}$ the cone in $M\otimes_{\BZ}\BR$ which is dual to $\sigma$. Then 
$$\BK[Y] \simeq \bigoplus_{m \in \sigma^{\vee}\cap M} \BK \chi^m \subseteq \BK[M],$$
where $\chi^m$~--- the character of $T$ which corresponds to $m\in M.$

We can consider the function $\varphi: m\rightarrow\sum\langle m,v_i\rangle$. This function induces a $\mathbb{Z}_{\geq 0}$-grading on $\BK[Y]$ with $\BK[Y]_0=\BK$:
$$\BK[Y]_i = \bigoplus_{m \in \sigma^{\vee}\cap M:\ \varphi(m) = i }\BK \chi^m.$$

Suppose there is a homogeneous LND $\partial$ of negative degree $k<0$ on $Y$. Then
$$\partial = \sum_{m \in M} {\partial_m}$$
where $\partial_m$ are derivations homogeneous with respect to $M$-grading of $\BK[Y]$ and $\varphi(m) = k$. The set $$P = \{ m\in M \mid \partial_m \neq 0\}$$
is finite. If we consider a vertex $p$ of a convex hull of $P$ in $M_{\BR}$ then $\partial_p$ is a LND.

An element $e \in M$ is called a \emph{Demazure root} if there is $i$ such that:
\begin{enumerate}
    \item $\langle e, v_i \rangle = -1;$
    \item $\langle e, v_j \rangle \geq 0$ when $j \neq i.$
\end{enumerate}

If $\partial_p$ is an $M$-homogeneous LND then by \cite[Theorem 2.7]{L} $p$ is a Demazure root. Then
$$\varphi(p)  = \sum_i \langle p, v_i \rangle = k <0.$$
It is possible only when there is $i$ such that $\langle p, v_i\rangle = -1$ and $\langle p, v_j \rangle = 0$ for all $j \neq i.$ But this implies that $Y\cong Z\times \BA^1$.

\item Nonrigid affine cones over smooth projective varieties. Such cones have a unique singular point. Hence, by Corollary~\ref{sinsin} they are of type~B.

\end{enumerate}

Examples of varieties of type C, i.e. rigid varieties are, for example, the following varieties. 
\begin{enumerate}
\item Toral varieties, i.e varieties $Y$ such that $\BK[Y]$ is generated by invertible functions, see for example~\cite{ShT}.
\item The affine Fermat cubic threefold $x^3+y^3+z^3+w^3=0$, see~\cite[Corollary~1.9]{ChPW}.
\item A criterion for non-normal toric variety to be rigid is given in~\cite{BG}. A toric variety is rigid if and only if the regular locus of it coincide with the open orbit of the torus. Note, that by~\cite{AKZ} each non degenerate (i.e. without nonconstant invertible functions) normal toric variety is flexible, and hence, it is not rigid.
\item Some more examples of rigid varieties are given in~\cite{CM2}, \cite{FM}, \cite{FMJ}.
\end{enumerate}

\section{Trinomial varieties}\label{do}

In this section we will consider trinomial varieties.

\begin{definition}\label{defTrin}\cite[Construction 1.1]{HW}

Fix integers $r, n > 0, m \geq 0$ and $q \in \{0, 1\}$. Also, fix a partition 
$$n = n_q + \ldots + n_r, \ n_i >0.$$
For each $i = q, \ldots, r$, fix a tuple $l_i = (l_{i1}, \ldots, l_{in_i})$ of positive integers and define a monomial
$$T_i^{l_i} = T_{i1}^{l_{i1}}\ldots T_{in_i}^{l_{in_i}} \in \mathbb{K}[T_{ij}, S_k | q \leq i \leq r, 1 \leq j \leq n_i, 1 \leq k \leq m].$$

We write $\mathbb{K}[T_{ij}, S_k]$ for the above polynomial ring. Now we define a ring $R(A)$ for some input data $A$. 

\emph{Type 1.} $q = 1, A = (a_1, \ldots, a_r)$ where $a_j \in \mathbb{K}$ with $a_i \neq a_j$ for $i\neq j$. Set $I = \{1, \ldots, r-1\}$ and for every $i \in I$ define a polynomial
$$g_i = T_i^{l_i} - T_{i+1}^{l_{i+1}} - (a_{i+1} - a_i) \in \mathbb{K}[T_{ij}, S_k].$$

\emph{Type 2.} $q = 0,$

$$A = \begin{pmatrix} 

a_{00} & a_{01} & a_{02} & \ldots & a_{0r}\\
a_{10} & a_{11} & a_{12} & \ldots & a_{1r}

\end{pmatrix}$$
is a $2 \times (r+1)$-matrix with pairwise linearly independent columns. Set $I = \{0, \ldots, r-2\}$ and for every $i \in I$ define a polynomial 
$$g_i = \mathrm{det}\begin{pmatrix} 
T_{i}^{l_i} & T_{i+1}^{l_{i+1}} & T_{i+2}^{l_{i+2}} \\
a_{0i} & a_{0i+1} & a_{0i+2} \\
a_{1i} & a_{1i+1} & a_{1i+2}
\end{pmatrix} \in \mathbb{K}[T_{ij}, S_k].$$

For both types we define $R(A) = \mathbb{K}[T_{ij}, S_k]/(g_i \mid i \in I).$ \emph{Trinomial variety} is a variety which is isomorphic to $\mathrm{Spec}\ R(A).$

\end{definition}

A criterion for trinomial variety to be rigid is given in~\cite{EGSh}.

\begin{proposition}\cite[Theorem 4]{EGSh}\label{rigidtrinomial}

Let $Y$ be a trinomial variety of type 1. Then $Y$ is not rigid if and only if one of the following holds:
\begin{enumerate}
\item $m > 0$;
\item There is $b \in \{1,\ldots, r\}$ such that for each $ i\in \{1, \ldots, r\} \setminus \{b\}$ there is $j(i) \in \{1, \ldots, n_i\}$ with $l_{ij(i)} = 1$.

\end{enumerate}

Let $Y$ be a trinomial variety of type 2. Then $Y$ is not rigid if and only if one of the following holds:
\begin{enumerate}

\item $m > 0$;

\item There are at most two numbers $a, b \in \{0,\ldots, r\}$ such that for each $i \in \{0,\ldots, r\} \setminus \{a, b\}$ there is $j(i) \in \{1, \ldots, n_i\}$ with $l_{ij(i)} = 1$. 

\item There are exactly three numbers $a,b,c \in \{0, \ldots, r\}$ such that for each $i \in \{a, b\}$ there is $j(i) \in \{1, \ldots, n_i\}$ with $l_{ij(i)} = 2$ and the numbers $l_{ik}$ are even for all $k \in \{1,\ldots, n_i\}$. Moreover, for each $i \in \{0, \ldots, r\}\setminus\{a,b,c\}$ there is $j(i) \in \{1, \ldots, n_i\}$ with $l_{ij(i)} = 1$.
\end{enumerate}

\end{proposition}

So we obtain that trinomial variety $Y$ is of type C if conditions 1)-2) for type 1 and conditions 1)-3) for type 2 of Proposition \ref{rigidtrinomial} do not hold.

Now suppose that $Y$ is a nonrigid trinomial variety of type 1. If $m>0$ then by Corollary~\ref{ba1} the variety $Y$ is of type~A. If $m = 0$ then for all $1\leq i\leq r$ except may be $i_0$ there exists $j=j(i)$ such that $l_{ij}=1$. We can consider linear combinations of equations that give $Y $, to obtain $T_i^{l_i}-T_j^{l_j}=a_j-a_i$. Since $a_i-a_j\neq 0$, the number of monomials $T_i^{l_i}$ which are zero in some point $y\in Y$ is not more than $1$. 

Let us fix any $j(i_0)$. One can define the following LND $\delta$ on $\BK[Y]$: 
$$
\delta(T_{ij(i)})=\prod_{k\neq i}\frac{\partial T_k^{l_k}}{\partial T_{kj(k)}}
$$
and $\delta(T_{il}) = 0$ if $l \neq j(i).$  For each $y\in Y$ there exists $i$ such that $\delta(T_{ij(i)})(y)\neq 0$. Therefore, $Y$ does not admit $\SAut(Y)$-fixed points. By Proposition~\ref{ppp} the variety $Y$ has type A.

Now suppose that $Y$ is a non-rigid trinomial variety of Type 2. Again, if $m> 0 $ the variety $Y$ is of type A. Suppose $m=0$. Also, we assume that there is no monomial $T_i^{l_i}=T_{i1}$ consisting of one variable. Otherwise, we can decrease $r$. Let us define the following $\mathbb{Q}$-grading: $\deg T_{ij}=\frac{1}{n_i}$. We can multiply all degrees in this grading by an appropriate positive integer number to obtain a $\mathbb{Z}$-grading. Let us denote this grading on~$R(A)$ by~$\eta$. This grading has a natural lifting to a grading $\xi$ on $\BK[T_{uv}]$.
For the grading~$\eta$ we have $\BK[X]_0=\BK$ and all negative components are zeros. We will prove that there are no nonzero LNDs with negative degrees, and then by the result from the previous section,~$X$ is of type B. Note that $\eta$-degrees of all monomials $T_i^{l_i}$ coincide. Let us denote this degree by $d$.    Suppose $\delta$ is a nonzero homogeneous LND of negative degree. Then the image 
$$\delta(T_i^{l_i})=\sum_{j=1}^{n_i} l_{ij}\frac{\partial T_i^{l_i}}{\partial T_{ij}}\delta(T_{ij})$$ 
of the monomial $T_i^{l_i}$ has degree less than $d$. Since $\delta$ is nonzero, there is $T_{ij}\notin\Ker\delta$. By~\cite[Principle~1(a)]{Fr}, the kernel of an LND is factorially closed, therefore, $T_i^{l_i}\notin\Ker\delta$. 
Since $\deg\delta(T_{pq})<\deg T_{pq}<d$, there exists a unique polynomial $f_{pq}\in \BK[T_{uv}]$ with $\xi$-degree less that $\deg T_{pq}$ such, that $\delta(T_{pq})=f_{pq}$ in $R(A)$. 
There is a relation $aT_i^{l_i}+bT_k^{l_k}+cT_s^{l_s}=0$ containing $T_i^{l_i}$. So, we have $a\delta(T_i^{l_i})+b\delta(T_k^{l_k})+c\delta(T_s^{l_s})=0$. This implies the following equality in $\BK[T_{uv}]$:
$$
a\sum_{j=1}^{n_i} l_{ij}\frac{\partial T_i^{l_i}}{\partial T_{ij}}f_{ij}+b\sum_{j=1}^{n_k} l_{kj}\frac{\partial T_k^{l_k}}{\partial T_{kj}}f_{kj}+c\sum_{j=1}^{n_s} l_{sj}\frac{\partial T_s^{l_s}}{\partial T_{sj}}f_{sj}=0.
$$
Since $\sum_{j=1}^{n_i} l_{ij}\frac{\partial T_i^{l_i}}{\partial T_{ij}}f_{ij}\neq 0$, we obtain that there are two polynomials of the form
$\frac{\partial T_i^{l_i}}{\partial T_{ij}}f_{ij}$ and $\frac{\partial T_k^{l_k}}{\partial T_{ku}}f_{ku}$ having a coinciding monomial $M$. Hence, $M$ is divisible by $\frac{\partial T_i^{l_i}}{\partial T_{ij}}\frac{\partial T_k^{l_k}}{\partial T_{ku}}$. So, we have
$$
d>\deg M\geq \deg\frac{\partial T_i^{l_i}}{\partial T_{ij}}+\deg\frac{\partial T_k^{l_k}}{\partial T_{ku}}\geq \frac{\deg T_i^{l_i}}{2}+\frac{\deg T_k^{l_k}}{2}=d.
$$
This give a contradiction.


So we obtain the following proposition.
\begin{proposition}
    Let $Y$ be a nonrigid trinomial variety. Then
    \begin{enumerate}
        \item if $m > 0$ then $Y$ is of type A;
        \item if $Y$ is of Type 1 then $Y$ is of type A;
        \item if $Y$ is of Type 2 and $m= 0$ then $Y$ is of type B.
    \end{enumerate}
\end{proposition}

\section{Conjectures and remarks}

Let us state some conjectures about modified Derksen invariant. 

We do not know any examples of varieties of type A that do not satisfy  conditions of Porpositions~\ref{pp} and~\ref{ppp}. So, let us state a conjecture.
\begin{conjecture}
    The following conditions are equivalent:
    \begin{itemize}
        \item a variety $Y$ is of type A;
        \item the ideal $I$ of $\BK[Y]$ generated by all images of all LNDs on $Y$ coincide with $\BK[Y]$;
        \item there are no $\SAut(Y)$-fixed points on $Y$.
    \end{itemize}
\end{conjecture}

Note that for nonrigid toric varieties we have even more. A toric variety $X$ is of type $A$ if and only if $Y$ admits an LND $\partial$ with a slice. Then image of $\partial$ contains $1$, i.e. $I=\BK[Y]$. 

It is interesting question what is the explicit form of $\HD^*(X)$ in case when $Y$ is of type~B.
\begin{conjecture}
Let $I$ be the ideal of $\BK[Y]$ generated by all images of all LNDs on $Y$. Then 
$$\HD^*(Y)=\BK[Y]\oplus\bigoplus_{i>0}Iu^i.$$
\end{conjecture}

Let us recall that an affine irreducible variety $Y$ is called {\it rigid factor} if $Y\times \BA^1\cong Z\times\BA^1$ implies $Y\cong Z$. I.e. $Y$ is a rigid factor if it is not a conter-examle to generalized Zariski cancellation problem. All varieties that we know to be not rigid factors are of type A. It is well known that varieties of type C are rigid factors,  see~\cite[Theorem~2.24]{Fr}.  Therefore, we would like to state the following question.
\begin{question}
    Is there a variety of type B which is not a rigid factor?
\end{question}

\end{document}